\documentclass[letterpaper,12pt,reqno]{amsart}
\usepackage[margin=1in]{geometry}
\usepackage{hyperref}
\hypersetup{
     colorlinks   = true,
     citecolor    = black,
     linkcolor    = blue
}

\usepackage{graphicx,amsmath,amssymb,amsthm,paralist,color,tikz-cd}
\usepackage{mathrsfs}
\usepackage{mathtools}
\usepackage{cite}
\usepackage{setspace}

\usepackage{amscd}

\usepackage{bbm} 

\newtheorem{theorem}{Theorem}[section]
\newtheorem{lemma}[theorem]{Lemma}

\newtheorem{proposition}[theorem]{Proposition}

\newtheorem{corollary}[theorem]{Corollary}

\theoremstyle{definition}
\newtheorem{definition}[theorem]{Definition}
\newtheorem*{definition-nono}{Definition}

\newtheorem{example}[theorem]{Example}

\newtheorem{remark}[theorem]{Remark}

\newtheorem*{acknowledgement}{Acknowledgement}

\newcommand{\N}{\mathbb{N}}
\newcommand{\Z}{\mathbb{Z}}

\newcommand{\R}{\mathbb{R}}

\newcommand{\mc}{\mathcal}

\newcommand{\mf}{\mathfrak}

\renewcommand{\a}{\alpha}
\renewcommand{\b}{\beta}
\newcommand{\g}{\gamma}
\newcommand{\G}{\Gamma}
\renewcommand{\d}{\delta}
\newcommand{\e}{\varepsilon}
\renewcommand{\l}{\lambda}

\newcommand{\s}{\sigma}
\newcommand{\vp}{\varphi}

\renewcommand{\th}{\theta}

\newcommand{\set}[1]{\left\{#1\right\}}
\renewcommand{\r}{\rightarrow}

\def\multiset#1#2{\ensuremath{\left(\kern-.3em\left(\genfrac{}{}{0pt}{}{#1}{#2}\right)\kern-.3em\right)}}
\newcommand{\norm}[1]{\left\lVert#1\right\rVert}

\numberwithin{equation}{section}

\subjclass[2010]{37A30, and 22F30}
\keywords{Pointwise Equidistribution, Birkhoff's Theorem, Ratner's Theorem}

\title[Pointwise Ergodic Theorems and Translated Measures]{Pointwise Equidistribution and Translates of Measures on Homogeneous Spaces}
\author{Osama Khalil}
\address{Department of Mathematics, The Ohio State University, Columbus, OH 43210}
\email{khalil.37@osu.edu}

\date{}

\begin{document}

\begin{abstract}
	Let $(X,\mf{B},\mu)$ be a Borel probability space.
    Let $T_n: X\r X$ be a sequence of continuous transformations on $X$.
    Let $\nu$ be a probability measure on $X$ such that
    $\frac{1}{N}\sum_{n=1}^N (T_n)_\ast\nu \r \mu$ in the weak-$\ast$ topology.
    Under general conditions, we show that for $\nu$ almost every $x\in X$, the measures $\frac{1}{N}\sum_{n=1}^N \d_{T_n x}$ get equidistributed towards $\mu$ if $N$ is restricted to a set of full upper density.
    We present applications of these results to translates of closed orbits of Lie groups on homogeneous spaces.
    As a corollary, we prove equidistribution of exponentially sparse orbits of the horocycle flow on quotients of $SL(2,\R)$, starting from every point in almost every direction.
\end{abstract}

\maketitle
{\hypersetup{linkcolor=black}}

\section{Introduction}

	Many problems in number theory and geometry can be recast in terms of the equidistribution of translates of appropriate measures on quotients of certain Lie groups.
    The general set up of these results is a Borel probability space $(X,\mf{B},\mu)$, a probability measure $\nu$ on $X$ (usually singular with respect to $\mu$) and a sequence of transformations $T_n: X\r X$ such that
    \begin{align} \label{convergence of nu to mu}
    	\frac{1}{N} \sum_{n=1}^N (T_n)_\ast \nu \xrightarrow{N\r\infty} \mu
    \end{align}
    where $(T_n)_\ast \nu$ is the pushforward of $\nu$ under $T_n$ and the convergence is in the weak-$\ast$ topology.
    A natural question is to what extent can one extend such results to describe the behavior of measures of the form
    \begin{align} \label{pointwise}
    	\frac{1}{N} \sum_{n=1}^N \d_{T_n x}
    \end{align}
    for $\nu$-almost every $x$, where $\d_y$ denotes the dirac delta measure at a point $y$.
    
    Recently, an analogous question for flows was addressed by Chaika and Eskin~\cite{EskinChaika} in the context of flat surfaces.
    In that setting, $X$ is some affine submanifold of the moduli space of flat structures on a surface, $\mu$ is a natural affine $SL(2,\R)$ invariant measure, $\nu$ is the measure supported on an orbit of $SO(2)$ and the transformations are of the form $ a(t) = diag(e^t , e^{-t})$.
    They show that for $\nu$ almost every $x$, the empirical measures analogous to those in~\eqref{pointwise} get equidistributed towards $\mu$.
    
    In the context of homogeneous spaces, Shi~\cite{RShi-Pointwise} explored this question for translates of measures supported on (pieces of) orbits of certain horospherical subgroups of Lie groups by one parameter diagonalizable subgroups.
    Here $X$ is a homogeneous space for a Lie group $G$, $\nu$ is a measure on an orbit of a certain horospherical subgroup, and $T_n = T^n$, where $T$ is an Ad-diagonalizable element of $G$.
    Equidistribution of empirical measures towards the natural $G$-invariant Haar measure is proven $\nu$ almost everywhere.
    
    In~\cite{KleinbockShiWeiss}, an effective version of this result is obtained via different methods.
    The convergence of empirical measures analogous to~\eqref{pointwise} is proven for general dynamical systems under the hypothesis of some form of exponential mixing of the transformation $T$ with respect to the non-invariant measure $\nu$.
    
    In all three cases, equidistribution was obtained by exploiting specific properties of the system at hand, while not directly utilizing the fact that~\eqref{convergence of nu to mu} holds.

    \subsection{Statement of Results}
    In this article, we approach this question in the general context of continuous measure preserving transformations, assuming~\eqref{convergence of nu to mu} only.
    We obtain equidistribution results of measures in~\eqref{pointwise} under general conditions yet only along subsequences of full upper density.
    Recall the definition of upper density:
    \begin{definition-nono} \nonumber
		The upper density of a subset $A \subseteq \N$, denoted by $\overline{d}(A)$ is defined to be
        	\begin{eqnarray*}
        		\overline{d}(A) = \limsup_{N\r\infty} \frac{\# (A\cap[1,N])}{N}
        	\end{eqnarray*}
	\end{definition-nono}

    In what follows, $X$ will be a locally compact, second countable topological space and $\mc{B}$ is its Borel $\s$-algebra. A pair $(X,\mc{B})$ will be called a standard Borel space.
    The following is our first main result for the case when translations are done by powers of a single transformation.
    
    \begin{theorem} \label{Intro-Birkhoff along a subsequence}
          Let $(X, \mc{B}, \mu)$ be a standard Borel probability space. Let $T$ be a continuous ergodic measure preserving transformation of $X$. 
          Let $\nu$ be a probability measure on $X$ such that
          \[ \frac{1}{N} \sum_{n=0}^{N-1} T^n_\ast \nu \xrightarrow[N\r\infty]{weak-\ast}\mu\]
          Then, for $\nu$-almost every $x\in X$, there exists a sequence $A(x) \subseteq \N$, of upper density $1$, such that
          for all $\psi \in C_c(X)$,
          \[ \lim_{\substack{N\r\infty \\ N\in A(x)}} 
            \frac{1}{N} \sum_{n=0}^{N-1} \psi(T^nx) = \int \psi \; d\mu \]
      \end{theorem}
      
      The proof of Theorem~\ref{Intro-Birkhoff along a subsequence} relies on an adaptation of the weak-type maximal inequality and follows similar lines to the proof of the classical Birkhoff ergodic theorem.
      
      Our next result concerns the more general situation of translating by sequences of transformations.
      In this generality, we assume more structure on the possible limit points of the empirical measures.
      It would be interesting to know if a similar statement is still valid under weaker, more general hypotheses.
      
      \begin{theorem} \label{Intro-Birkhoff for sequences of transformations}
          Let $(X, \mc{B}, \mu)$ be a standard Borel probability space. Let $(T_n)_n$ be a sequence of continuous transformations of $X$.
          Let $S:X\r X$ be a continuous ergodic $\mu$ preserving transformation.
          Let $\nu$ be a probability measure on $X$.
          Assume the following holds:
          \begin{enumerate}
            \item $\frac{1}{N} \sum_{n=1}^{N} (T_n)_\ast \nu \xrightarrow{N\r\infty}  \mu$ in the weak-$\ast$ topology.
            \item For $\nu$-almost every $x\in X$, any limit point of the sequence of measures 
            $\frac{1}{N} \sum_{n=1}^{N} \d_{T_nx} $ is $S$-invariant.
            \item There exists a Borel measurable, $\s$-compact set $Z\in \mc{B}$ such that
            $\mu(Z) = 0$ and for all ergodic $S$ invariant measures $\l \neq \mu$, $\l(Z) = 1$.

          \end{enumerate}

          Then, for $\nu$-almost every $x\in X$, there exists a sequence $A(x) \subseteq \N$, of upper density equal to $1$, such that for all $\psi \in C_c(X)$,
          \[  
            \lim_{\substack{N\r\infty\\ N\in A(x)}} \frac{1}{N} \sum_{n=1}^{N} \psi(T_nx) = \int \psi \; d\mu \]
      \end{theorem}
      We remark that the maximal inequality used in the proof of Theorem~\ref{Intro-Birkhoff for sequences of transformations} does not require the hypotheses of the Theorem.


	We now discuss some applications of these results.
    Theorem~\ref{Intro-Birkhoff along a subsequence} is broadly applicable to general dynamical systems and thus
    we present some applications of Theorem~\ref{Intro-Birkhoff for sequences of transformations}
    within homogeneous dynamics where we demonstrate that its hypotheses are verified.

\subsection{Sparse Equidistribution and Translates of Orbits of Lie Groups}
      
      Our motivation for studying this question comes from the problem of sparse equidistribution of unipotent flows on homogeneous spaces.
      This was conjectured by Shah in~\cite{ShahPolynomial}.
      Recent progress was achieved in~\cite{Venkatesh:Equi} for the horocycle flow on compact quotients of $SL_2(\R)$ along sequences of the form $n^{1+\g}$ for small values of $\g$.
      See also the work of Sarnak and Ubis on the equidistribution along the primes~\cite{Sarnak2015} on $SL_2(\R)/SL_2(\Z)$.
      However, the question in full generality remains open.
      
      With the help of Theorem~\ref{Intro-Birkhoff for sequences of transformations}, we obtain an equidistribution result for exponentially sparse orbits of unipotent flows, of which the horocycle flow is an example.
      We will need some notation to state our results precisely.
      
      Let $G$ be a connected semisimple Lie group and let $\G$ be a lattice in $G$.
      Let $\mu_{G/\G}$ denote the unique $G$ invariant Haar probability measure on $G/\G$.
      Let $H$ be a closed connected subgroup such that the orbit $H \G$ is
      closed in $G/\G$ and supports a probability $H$-invariant measure, denoted by $\mu_H$.
      
      In order to apply Theorem~\ref{Intro-Birkhoff for sequences of transformations} in this setup, we introduce the notion of Ratner Sequences.      
      We say a sequence $g_n$ of elements of $G$ is a \textbf{Ratner Sequence} with respect to $H$ if there exists a non-trivial $Ad$-unipotent element
      $u \in G$ such that for $\mu_H$ almost every $x\in G/\G$,
      any limit point of the sequence of empirical measures $\frac{1}{N} \sum_{n=1}^N \d_{g_n x}$ is invariant by $u$.
      
	  Our main theorem in this set up is the following.
      \begin{theorem} \label{Intro-Pointwise equidistribution for general groups}
        Let $g_n$ be a Ratner Sequence with respect to $H$ satisfying
        \begin{equation} \label{eqn: equidist assumption}
         \frac{1}{N} \sum_{n=1}^N (g_n)_\ast \mu_H \xrightarrow{N\r\infty}
        	\mu_{G/\G}
        \end{equation}        
        Then, for $\mu_{H}$ almost every $x\in G/\G$, there exists a sequence $A(x) \subseteq \N$ of upper density $1$ such that
        \[ \lim_{\substack{N\r\infty\\N\in A(x)}} \frac{1}{N} \sum_{n=1}^N \d_{g_n x} = \mu_{G/\G} \]        
      \end{theorem}

      Apart from Theorem~\ref{Intro-Birkhoff for sequences of transformations}, a key ingredient in the proof of Theorem~\ref{Intro-Pointwise equidistribution for general groups} is Ratner's measure classification theorem~\cite{RatnerMeasure}.
      
      The assumption on the equidistribution of the translates of $\mu_H$ by a sequence of elements $g_n$ is satisfied in numerous examples in homogeneous dynamics.
      For example, it was shown in~\cite{eskin1993} that~\eqref{eqn: equidist assumption} is satisfied as soon as $H$ is a symmetric subgroup of $G$ and $g_n$ tends to infinity in $G/H$.
      Recall that $H$ is said to be a symmetric subgroup if it is the fixed point set of an involution of $G$.
      
      The result of Eskin and McMullen was significantly strengthened in~\cite{EskinMozesShah} to include translates of reductive subgroups that is, subgroups which are only invariant by an involution of $G$.
      Note that we only require $H$ to be connected in Theorem~\ref{Intro-Pointwise equidistribution for general groups}.
      Thus, the applicability of Theorem~\ref{Intro-Pointwise equidistribution for general groups} boils down to the existence of Ratner sequences which is the subject of the next section.

      \subsubsection{\textbf{Existence of Ratner Sequences}}
      We prove a general criterion on the existence of sparse Ratner Sequences inside $1$-parameter unipotent subgroups.
      The proof relies on the representation theory of embedded copies of $SL_2(\R)$ inside $G$ in addition to an adaptation of a technique developed by Chaika and Eskin~\cite{EskinChaika} in the setting of strata of Abelian differentials.

      Let the notation be the same as above and let $Z_G(H)$ denote the centralizer of $H$ inside $G$.
      The following is the main statement.

      \begin{theorem} \label{thrm: existence of Ratner sequences in 1-param unipotent subgroups}
    Let $U = \set{u_t: t\in \R}$ be a $1$-parameter $Ad$-unipotent subgroup of $G$ such that $U\not\subset Z_G(H)$. 
	Suppose $t_n >0$ is a sequence satisfying
        	\[  t_n = O\left(  e^{\l n } \right)  \]
    for some constant $\l > 0$ and all $ n \geq 1$,
    Then, $u_{t_n}$ is a Ratner sequence for $H$.
    \end{theorem}
    The exponential growth condition is used in the adaptation of Chaika-Eskin's technique which is similar in spirit to Breiman's law of large numbers.
    It uses the law of large numbers for weakly dependent random variables.
    Exponential growth guarantees a sufficiently fast rate of decay of correlations.
      
      In the appendix, we prove a more crude criterion for the existence of Ratner sequences in the setting where $H$ is a symmetric subgroup of $G$, but we drop the restriction on the unipotence of the elements $g_n$.
           
      We show that any sequence $g_n$ satisfying an exponential growth condition similar to the one in Theorem~\ref{thrm: existence of Ratner sequences in 1-param unipotent subgroups} contains a Ratner sequence as a subsequence.
      See Theorem~\ref{unipotent invariance} for the precise statement.

      In the particular instance when $G=SL_2(\R)$ and the sequence $g_n$ comes from the action of the horocycle flow,      
      Theorem~\ref{Intro-Pointwise equidistribution for general groups} takes the following form:
      
      \begin{corollary} \label{sparse equidistribution corollary}
      Let $G = SL(2,\R)$, $\G\subset G$ a lattice and let $K = SO(2)$.
      Let $\l>0$ and for $n\in\N$, let $t_n = e^{\l n}$.
      Let
      \[ g_n = \begin{pmatrix}
      1 & t_n \\ 0 & 1
      \end{pmatrix},
      k_\th = \begin{pmatrix}
      	\cos(\th) & \sin(\th) \\
        -\sin(\th) & \cos(\th)
      \end{pmatrix} 
      \]

      Then, for \textit{every} $x\in G/\G$ and for almost every $\theta\in [0,2\pi]$, there exists a sequence $A(\theta)\subseteq \N$ of upper density $1$ such that
      \[
      \lim_{\substack{N\r\infty\\N\in A(\theta)}} \frac{1}{N} \sum_{n=1}^N \d_{g_n k_\th x} = \mu_{G/\G}
      \]
      Moreover, if $G/\G$ is compact then $A(\theta) = \N$.
      \end{corollary}
      
      \begin{proof}
      Let $H = SO(2)$.
      Since $H$ is its own centralizer in $G$, by Theorem~\ref{thrm: existence of Ratner sequences in 1-param unipotent subgroups}, $g_n$ is a Ratner sequence for $H$.
      By the work of Eskin and McMullen~\cite{eskin1993}, one has that $g_n \mu_H \r \mu_{G/\G}$.
      Thus, the corollary follows from Theorem~\ref{Intro-Pointwise equidistribution for general groups}.
      When $G/\G$ is compact, the action by unipotent elements on $G/\G$ is uniquely ergodic in this case, the Haar measure being the only invariant measure.
    Moreover, the space of probability measures on $G/\G$ is weak-$\ast$ compact and hence all limit points of the sequences of empirical measures are probability measures.
    Since $g_n$ is a Ratner sequence, the claim follows.
      \end{proof}
      
      The main point is that this holds for \text{every} $x$.
      This is not guaranteed by any general theorem on sparse equidistribution almost everywhere.


The paper is organized as follows. In \S~\ref{Maximal Ergodic Theorem Section}, we prove an analogue of the maximal ergodic theorem in our set up.
We use this to prove Theorems~\ref{Intro-Birkhoff along a subsequence} and~\ref{Intro-Birkhoff for sequences of transformations}
in \S~\ref{section Birkhoff for powers of a single transformation} and \S~\ref{section Birkhoff for sequences of transformations}.
In \S~\ref{section unipotent invariance}, we discuss the existence of Ratner sequences and prove Theorem~\ref{thrm: existence of Ratner sequences in 1-param unipotent subgroups}.
In \S~\ref{section-general groups}, we provide a proof of Theorem~\ref{Intro-Pointwise equidistribution for general groups}.

\begin{acknowledgement}
	I would like to thank my advisor, Nimish Shah, for numerous valuable discussions.
\end{acknowledgement}

\section{An Analogue of the Maximal Inequality}
\label{Maximal Ergodic Theorem Section}

	The following proposition is an extension of the classical maximal ergodic theorem to the set up involving sequences of transformations and a non-invariant measure.
    
  \begin{proposition} \label{Weak Maximal Ergodic Theorem}
      Let $(X, \mc{B})$ be a standard Borel probability space. Let $T_n$ be a sequence of continuous transformations on $X$. Let $\nu,\mu$ be probability measures on $X$ such that
      \[ \frac{1}{N} \sum_{n=0}^{N-1} (T_n)_\ast \nu \xrightarrow[N\r\infty]{weak-\ast}\mu\]
      Let $\psi \in C_c(X)$ and let $\a >0$ and $\b\in (0,1)$. For every $j,N\geq 1$, define the set
      \[ E_{\a,N,j}^\psi = \set{x\in X: \sup_{1\leq M\leq N } \left|\frac{1}{M} \sum_{k=j}^{j+M-1} \psi(T_k x)\right| > \a   } \]

      Then, for every $N\gg 1$, depending on $\psi$, there exists some $0\leq j_N< \b N$, such that

      \[ \a\b \nu( E_{\a,N,j_N}^\psi) \leq 12 ||\psi||_{L^1(\mu)} \]

  \end{proposition}

	We will deduce this proposition from the classical maximal inequality for $l^1(\Z)$ which is a consequence of Vitali's covering lemma. The proof follows closely the proof of the usual maximal inequality.
    
    \begin{lemma} [Lemma 2.29,~\cite{EinsiedlerWard}]  \label{max inequality for shift}
    	Let $\phi \in l^1(\Z)$. Define the following maximal function, for $a\in \Z$:
        	\[ \phi^\ast(a) = \sup_{N\geq 1}\left| 
            \frac{1}{N} \sum_{i=0}^{N-1} \phi(i+a) \right|	\]
         Let $\a>0$ and define 
         	\[ E_\a =\set{a\in\Z\mid \phi^\ast(a) > \a} \]
         Then,
         \[ \a \left| E_\a\right| \leq 3 ||\phi||_{l^1(\Z)} \]
    \end{lemma}

    \begin{proof}[Proof of Proposition~\ref{Weak Maximal Ergodic Theorem}]
    	Let $\psi \in C_c(X)$ and let $\a>0$, $\b\in (0,1)$.
        Let $N\geq 1$ and let $E_{\a,N,j}$ be as in the statement.
        Let $x\in X$ and let $J > N$ be a parameter to be determined later. Define the following function
        \[
        	\phi(j) = \begin{cases}
        		\psi(T_j x) & 0\leq j \leq J\\
                0 & otherwise
        	\end{cases}
        \]
        
        Then, clearly $\phi \in l^1(\Z)$. For $a\in \Z$, define the following two functions
        \begin{align*}
          \phi^\ast(a) =  \sup_{1\leq M} \left|\frac{1}{M}\sum_{k=0}^{M-1}\phi(k+a)\right|, &&
          \phi^\ast_N (a)= \sup_{1\leq n\leq N} 
                  \left|\frac{1}{n}\sum_{k=0}^{n-1}\phi(k+a)\right|
        \end{align*}
        
        Define the corresponding exceptional sets
        \begin{align*}
        	E_{\a}^\phi = \set{a\in \Z \mid \phi^\ast(a) > \a}, &&
            E_{\a,N}^\phi = \set{a\in [0,J-N-1] \mid \phi_N^\ast(a) > \a}
        \end{align*}
        
        By Lemma~\ref{max inequality for shift} applied to $\phi$, we have
        
        \begin{eqnarray} \label{shift inequality}
        	\a\left| E_{\a,N}^\phi \right| \leq \a \left| E_\a^\phi \right| \leq 3||\phi||_{l^1(\Z)}
        \end{eqnarray}
        
        Note that for $a\in [0,J-N-1]$, we have
        \begin{align} \label{phi = psi}
        \phi^\ast_N(a) = \sup_{1\leq n\leq N} \left|\frac{1}{n}\sum_{k=0}^{n-1}\psi(T_{k+a}x)\right|
        \end{align}
        
        Let $\chi_{{\a,N,j}}$ denote the indicator function of $E_{\a, N,j}^\psi$.
        Thus, for $j \in [0,J-N-1]$,
        \begin{align} \label{relationship between exceptional set for psi and phi}
          \chi_{{\a,N,j}}(x) =1 \text{ if and only if } j \in E_{\a,N}^\phi
        \end{align}

        Thus, combining~\eqref{shift inequality},~\eqref{phi = psi} and~\eqref{relationship between exceptional set for psi and phi} along with the definition of $\phi$, we get
        \begin{align*}
        	\a \sum_{j=0}^{J-N-1} \chi_{{\a,N,j}}(x)
            	&= \a \left| E_{\a,N}^\phi \right|
                \leq 3 \sum_{j=0}^{J} |\phi(j)| = 3 \sum_{j=0}^J |\psi(T_jx)|
        \end{align*}
        
        Integrating both sides of the above with respect to $\nu$ yields
        \begin{eqnarray} \label{equality on average}
        	\a \sum_{j=0}^{J-N-1} \nu(E_{\a,N,j}^\psi) \leq 3\sum_{j=0}^J \int |\psi(T_jx)|\; d\nu(x)
        \end{eqnarray}
        
        Taking $J = (1+\b)N$ in~\eqref{equality on average} and dividing both sides by $J-N$
        \begin{align} \label{J = 1 +beta}
        	\frac{\a}{\b N} \sum_{j=0}^{\b N-1} \nu(E_{\a,N,j}^\psi) &\leq 3\frac{(1+\b)N+1}{\b N}  \frac{1}{(1+\b) N+1}\sum_{j=0}^{(1+\b) N} \int |\psi(T_jx)|\; d\nu(x) \nonumber \\
            &\leq \frac{6}{\b} \frac{1}{(1+\b) N+1}\sum_{j=0}^{(1+\b) N} \int |\psi(T_jx)|\; d\nu(x)
        \end{align}
        
        Now, by assumption,
        \begin{eqnarray} \label{convergence weak star}
        	\frac{1}{(1+\b) N+1}\sum_{j=0}^{(1+\b) N} \int |\psi(T_jx)|\; d\nu(x) 
            \r \int |\psi|\;d\mu  = ||\psi||_{L^1(\mu)} \nonumber
        \end{eqnarray}
        
        Thus, for all $N$ sufficiently large, depending only on $\psi$, we have
        \[
        	\frac{1}{(1+\b) N+1}\sum_{j=0}^{(1+\b) N} \int |\psi(T_jx)|\; d\nu(x) \leq 2 ||\psi||_{L^1(\mu)}
        \]
        
        Combining this with~\eqref{J = 1 +beta}, we get for all $N$ sufficiently large,
        \begin{align*} \label{average is small}
        	\frac{1}{\b N} \sum_{j=0}^{\b N-1} \nu(E_{\a,N,j}^\psi)
            &\leq \frac{12||\psi||_{L^1(\mu)}}{\a\b}
        \end{align*}
        
        Thus, there must exist some $j = j(N) \in [0,\b N-1]$ for which the conclusion of the Proposition holds.

    \end{proof}

\section{An Analogue of Birkhoff's Ergodic Theorem}
\label{section Birkhoff for powers of a single transformation}
	This section is dedicated to the proof of Theorem~\ref{Intro-Birkhoff along a subsequence}.
    With the maximal inequality for the non-invariant measure $\nu$ in place (Proposition~\ref{Weak Maximal Ergodic Theorem}), our proof will follow the same lines as Bourgain's approach to deduce the classical Birkhoff theorem from the mean ergodic theorem
    (cf.~\cite{BourgainArithmeticSets}, Section 2-C).

      Recall that for a sequence of sets $A_n$, the $limsup$ of these sets is the set of elements which belong to $A_n$ for infinitely many $n$.
      More precisely,
      \[ \limsup_{n\r\infty} A_n = \bigcap_{n\geq 1} \bigcup_{k\geq n} A_k\]

      The following simple observation will be used repeatedly in what follows.
      \begin{lemma} \label{lemma on lower bound of limsup set}
          Let $X$ be a standard Borel space and let $\mu$ be a probability measure on $X$.
          Let $A_n\subseteq X$ be a sequence of measurable sets such that $\mu(A_n) \geq \a$ for some $\a\in [0,1]$.
          Then,
          \[ \mu\left(\limsup_{n\r\infty} A_n\right) \geq \a  \]

      \end{lemma}

      \begin{proof}
          This follows from the definition of $\limsup_{n\r\infty} A_n$ as a decreasing intersection and the continuity of the measure $\mu$.
      \end{proof}

      The following Lemma is the main step in the proof of Theorem~\ref{Intro-Birkhoff along a subsequence}.

      \begin{lemma} \label{Birkhoff for finitely many functions}
          Let $(X, \mc{B}, \mu)$ be a standard Borel probability space. Let $T$ be an ergodic measure preserving transformation on $X$. Let $\nu$ be a probability measure on $X$ such that
          \[ \frac{1}{N} \sum_{n=0}^{N-1} T^n_\ast \nu \xrightarrow[N\r\infty]{weak-\ast}\mu\]
          Let $f_1,\dots,f_n \in C_c(X)$.
          Then, for $\nu$-almost every $x\in X$, there exists a sequence $A \subseteq \N$, of upper density $1$,
          depending on $x$ and the functions $f_1,\dots,f_n$,
          such that for all $k = 1,\dots,n$,
            \[ \lim_{\substack{N\in A \\N\r\infty}} 
              \frac{1}{N} \sum_{n=0}^{N-1} f_k(T^nx) = \int f_k \mu \]
      \end{lemma}

    Let us deduce Theorem~\ref{Intro-Birkhoff along a subsequence} from this Lemma first.

	\subsection{Proof of Theorem~\ref{Intro-Birkhoff along a subsequence}}
    
    Let $\mc{F}=\set{f_k \in C_c(X): k\in\N}$ be an enumeration of a countable set of continuous functions which are dense in $C_c(X)$ in the uniform norm.
	Then, it suffices to show that for $\nu$ almost every $x\in X$, there exists a sequence $A(x) \subseteq \N$, of full upper density such that for all $f\in \mc{F}$,
    \begin{align} \label{theorem conclusion}
    	\lim_{\substack{N\r\infty\\N\in A(x)}} \frac{1}{N}\sum_{k=1}^N f(T^kx) = \int f \; d\mu
    \end{align}
    
    For each $n\in \N$, let $\mc{F}_n = \set{f_1,\dots, f_n}\subset \mc{F}$.
    By Lemma~\ref{Birkhoff for finitely many functions}, for each $n$, there exists a set $X_n$ with $\nu(X_n) = 1$, such that for all $x\in X_n$, there exists a sequence $A(x,\mc{F}_n) \subseteq \N$, along which the limit in~\eqref{theorem conclusion} holds for all $f \in \mc{F}_n$.
    
    Let $Y = \cap_n X_n$. Then, $\nu(Y) = 1$. Let $y\in Y$.
    We will build a sequence $A(y)$ by induction from the sequences $A(y,\mc{F}_n)$ via a standard argument which we will make use of several times later.
    
    For each $n\in \N$, let $N_n\in \N$ be such that for all $N \geq N_n$ with $N\in A(y,\mc{F}_n)$, and all $f\in \mc{F}_n$,
    \begin{align} \label{choice of N_n}
    \left| \frac{1}{N}\sum_{k=1}^N f(T^ky) - \int f \; d\mu \right| &\leq \frac{1}{n}
    \end{align}

    Let $M_1 =N_1 $. If $M_j$ has been defined, let $M_{j+1}$ be such that the following holds
    \begin{align*}
      \frac{|A(y,\mc{F}_{j}) \cap [1,M_{j+1}]|}{M_{j+1}} &\geq 1 - \frac{1}{j}\\
      \frac{M_{j}}{M_{j+1}} &\leq \frac{1}{j} \\
       M_{j+1} &\geq N_{j+1}
    \end{align*}
    
    Note that the above implies that
    \[ \frac{|A(y,\mc{F}_{j}) \cap [M_j,M_{j+1}]|}{M_{j+1}} \geq 1 - \frac{2}{j}\]
    
    Now, define the sequence $A(y)$ as follows:
    \begin{align*}
    	A(y) = \bigcup_{j=1}^\infty A(y,\mc{F}_{j}) \cap [M_j,M_{j+1}]
    \end{align*}
    
    Thus, by construction, the upper density of $A(y)$ is equal to $1$.
    Now, let $f\in \mc{F}$. Then, $f\in \mc{F}_n$ for all $n\geq n_0$, for some $n_0\in\N$.
    Thus, for all $N \geq M_{n_0}$ such that $N\in A(y)$, there exists $j\geq n_0$, such that $N \in A(y,\mc{F}_j)\cap [M_j,M_{j+1}]$.
    Thus, since $M_j\geq N_j$, by~\eqref{choice of N_n}, the conclusion follows.


	\subsection{Proof of Lemma~\ref{Birkhoff for finitely many functions}}

		For any function $\psi$ and for every $N\geq 1$, let $\mu(\psi) = \int \psi \;d\mu$ and let 
    	\[A_N(\psi) =  \frac{1}{N} \sum_{n=0}^{N-1}\psi \circ T^n \]
        
        Let $\e \in (0,1)$.
        By Von Neumann's mean ergodic theorem, for all $k$, 
    	\[ A_N(f_k) \xrightarrow{L_1(\mu)} \mu(f_k) \]
        
        Hence, we can find some $M\gg 1$, for all $k \leq n$,
    	\[ \int |A_M(f_k) -\mu(f_k)| \; d\mu < \frac{\e^3}{n} \]
        
        Let $\b = \frac{\e}{C}$, where
        \[C = \max_{1\leq k\leq n} 2||f_k||_{L^\infty} +1\]
        
        Let $g_k =A_M(f_k) -\mu(f_k)$. Note that $||g_k||_\infty \leq C$.
        For all $k$ and for every $N\in \N$, define 
        \[ E_{\e,N}^k = \set{x\in X: \sup_{1\leq m\leq N } \left|\frac{1}{n} \sum_{l=0}^{m-1} g_k(T^l x)\right| > \e   } \]
		
        Then, by the analogue of the maximal inequality, Proposition~\ref{Weak Maximal Ergodic Theorem}, applied to $g_k$,
        the sequence of transformations $T_l = T^l$
        and $E_{\e,N,j} = T^{-j}E_{\e,N}^k$, for all $N$ sufficiently large, depending on $\e$, there exists $j_{N,k}\in [0,\b N]$ such that
        \begin{align} \label{measure of exceptional set}
            \nu(T^{-j_{N,k}} E_{\e,N}^k) \leq \frac{12 ||g_k||_{L^1(\mu)}}{\e\b} 
            \leq \frac{12C \e}{n}
        \end{align}
        
        Let $G_{N,k}^\e = X \setminus T^{-j_{N,k}} E_{\e,N}^k$ and let
        $  G_N^\e = \bigcap_{k=1}^n G_{N,k}^\e  $.
        Thus, by~\eqref{measure of exceptional set} and Lemma~\ref{lemma on lower bound of limsup set},
        \begin{align} \label{measure of limsup of G_N,epsilon}
        	\nu\left( \limsup_{N\r\infty} G_N^\e\right) \geq 1- 12C\e
        \end{align}
        
        Now, let $y\in G_N^\e$ and let $Q\in [\sqrt{\e} N,N]$.
        Then, for all $k=1,\dots,n$,
        by definition of $E_{\e, N}^k$ and our choice of $\beta$,
        \begin{align} \label{1 + beta computation}
            \left|A_{(Q+\b N)}(g_k)(y) \right| 
                &\leq \left| \frac{Q}{Q+\b N} \frac{1}{Q}\sum_{l=j_{N,k}}^{j_{N,k}+Q-1} g_k(T^ly)
                \right|
                + \frac{\b N ||g_k||_{L^\infty}}{Q+\b N}
                 \nonumber \\
                &\leq \left| A_Q(g_k)(T^{j_{N,k}}y)\right|
                + \frac{C\b}{\sqrt{\e}} \nonumber \\
                &\leq \e + \sqrt{\e} \leq 2\sqrt{\e}
        \end{align}
        
        Hence, in particular, for any $y \in \limsup_N G_N^\e$, there exists a sequence $N_i \r \infty$ for which~\eqref{1 + beta computation} holds for all $Q\in [\sqrt{\e} N_i, N_i]$ and for all $k = 1,\dots,n$.
        Define the following sequence for $y\in \limsup_N G_N^\e$
        \begin{align} \label{definition of A(y,epsilon)}
        	A(y,\e) = \bigcup_{N_i:y\in G_{N_i}^\e} [(\sqrt{\e}+\b) N_i, (1+\b) N_i] \cap \N
        \end{align}
        
        Now, a simple computation shows that for all $N$, and any function $\psi$,
         \begin{align} \label{average of average}
            A_N(A_M(\psi)) &= \frac{1}{NM} \sum_{n=0}^{N-1}\sum_{m=0}^{M-1} \psi \circ T^{n+m} \nonumber \\
            &= A_N(\psi) + O_M\left(\frac{||\psi||_{L^\infty}}{N}\right)
         \end{align}

         Combining~\ref{1 + beta computation} and~\ref{average of average} implies that for every $y\in \limsup_N G_N^\e$,
         all $k\leq n$ and for all $Q\in A(y,\e)$ such that $Q \gg M$,
        \begin{align} \label{limsup inequality}
            \left| A_{Q}(f_k)(y) - \mu(f_k)\right| \leq 3\sqrt{\e} 
        \end{align}

        Note that the choice of $\e$ in the above considerations was arbitrary.
		Hence, we may consider sets $B^{\e_i} = X \setminus \limsup_N G_N^{\e_i}$, where $\e_i = 1/i^2$.
        Then, by~\eqref{measure of limsup of G_N,epsilon} and the Borel-Cantelli lemma applied to $B^{\e_i}$,
        \begin{align} \label{desired set of full measure}
        	\nu\left(X\setminus \limsup_{i\r\infty} B^{\e_i}\right) = 1
        \end{align}

		Let $y\in X\setminus \limsup_i B^{\e_i}$.
        We claim that there exists a sequence $A(y) \subseteq \N$ of full upper density, such that for all $k\leq n$,
        \begin{align}  \label{lemma conclusion}
        	\limsup_{\substack{N\r\infty\\ N\in A(y)}} \left| A_{N}(f_k)(y) - \mu(f_k)\right| = 0
        \end{align}
        By~\eqref{desired set of full measure}, this will conclude the proof.
        
        Since $ y \in \limsup_N G_N^{\e_i}$ for all $i$ sufficiently large, we have sequences $A(y,\e_i)$ as before for all $\e_i$ sufficiently small.
        Without loss of generality, we may assume this holds for all $\e_i$.
        Note that by~\eqref{definition of A(y,epsilon)}, the upper density of $A(y,\e_i)$ is at least $\frac{1-\sqrt{\e_i}}{1+\e_i/C}$.
        
        We build the sequence $A(y)$ from $A(y,\e_i)$ by induction as follows.
        Let $N_i \in A(y,\e_i)$ be such that~\eqref{limsup inequality} holds for all $k\leq n$ and all $Q\geq N_i$.
        Let $M_1 = N_1$.
        If $M_j$ is defined, let $M_{j+1} \in \N$ be such that
        \begin{align*}
          \frac{|A(y,\e_{j}) \cap [1,M_{j+1}]|}{M_{j+1}} &\geq \frac{1-\sqrt{\e_{j}}}{1+\e_{j}/C} - \frac{1}{j}\\
          \frac{M_{j}}{M_{j+1}} &\leq \frac{1}{j} \\
          M_{j+1} &\geq N_{j+1}
    	\end{align*}
    
        This in particular, implies that
        \[	\frac{|A(y,\e_{j}) \cap [M_j,M_{j+1}]|}{M_{j+1}} \geq \frac{1-\sqrt{\e_{j}}}{1+\e_{j}/C} - \frac{2}{j} \]
        
        Now, define the sequence $A(y)$ as follows:
        \[  A(y) = \bigcup_{j=1}^\infty ( A(y,\e_{j}) \cap [M_j,M_{j+1}] ) \]
        
        Thus, since $\e_j \r 0$, the upper density of $A(y)$ is equal to $1$.
        Moreover, by~\eqref{limsup inequality} and by choice of $M_j$, we have that~\eqref{lemma conclusion} holds as desired.

\section{An Analogue of Birkhoff's Theorem for Sequences of Transformations}
\label{section Birkhoff for sequences of transformations}

	In this section, we prove Theorem~\ref{Intro-Birkhoff for sequences of transformations}.
    We will use similar ideas to those used in the proof of Theorem~\ref{Intro-Birkhoff along a subsequence} by applying the weak-type maximal inequality to a carefully chosen set of continuous functions capturing the structure of the ergodic invariant measures under the transformation $S$.

  	First, we make some standard reductions.
    Note that since $X$ is locally compact and second countable, by passing to the one point compactification and extending all the transformations on $X$ trivially to the point at infinity, we may assume that $X$ is in fact compact.
    
    The set $Z$ in the assumption will then be enlarged to include the point at infinity since the dirac measure at that point will be an ergodic invariant measure for $S$.
    Also, since $X$ is now assumed compact, the space of probability measures on it is weak-$\ast$ compact and thus we can always find limit points of infinite sequences.

    \begin{proof}[Proof of Theorem~\ref{Intro-Birkhoff for sequences of transformations}]
    Let $\e \in (0,1)$ be fixed and let $\e_n = \e^2/4^n$ for $n\in \N$.
    Write $Z = \cup_n K_n$, where $K_n$ is compact and $K_n \subseteq K_{n+1}$ for each $n$.
    By regularity of the measure $\mu$, since $\mu(Z) = 0$, there exists an open set $U_n$ containing $Z$, such that $\mu(U_n\setminus K_n) < \e_n$, for each $n$.
    
    Moreover, by Urysohn's lemma, we can find a continuous function $0\leq f_n\leq 1$ such that $f_n|_{K_n} \equiv 1$ and $f_n \equiv 0$ on $X\setminus U_n$.
    Thus,
    \begin{align*}
    	||f_n||_{L^1(\mu)}= \int f_n(x) \;d\mu(x) < \e_n
    \end{align*}
    
	Let $n$ be fixed. For each $j\in \N$, $k\leq n$ and $\a\in \R$, define the following set
    \[
    	E_{\a,N,j}^k = \set{x\in X: \sup_{1\leq M\leq N } \frac{1}{M} \sum_{m=j+1}^{j+M} f_k(T_m x) > \a   }
    \]
    
    Applying the analogue of the maximal inequality, Proposition~\ref{Weak Maximal Ergodic Theorem}, with $f_k$, $\a_k = \b_k = \e_k^{1/4}$, we get that for all $N$ sufficiently large, depending on $f_k$, there exists $j_{N,k}\in [0,\b_k N]$ such that
    \begin{align} \label{measure of exceptional set for one N}
    	\nu\left(E_{\a_k,N,j_{N,k}}^k \right) \leq \frac{12 ||f_k||_{L^1(\mu)}}{\a_k\b_k} \ll \e_k^{1/2} 
    \end{align}
    for each $k\leq n$.
    
    Let $G_{N,k} = X \setminus E_{\a_k,N,j_{N,k}}^k$. Let $y\in G_{N,k}$ and let $Q\in [\e_k^{1/8} N,N]$.
    Then, by definition of $E_{\a_k,N,j_N}^k$,
    \begin{align} \label{estimate of ergodic sum}
    	\frac{1}{Q+\b_k N} \sum_{l=1}^{Q+\b_k N} f_k(T_l y) 
        	&\leq \frac{Q}{Q+\b_k N} \frac{1}{Q}\sum_{l=j_{N,k}+1}^{j_{N,k}+Q} f_k(T_ly)
            + \frac{\b_k N ||f_k||_{L^\infty}}{Q+\b_k N}
             \nonumber \\
        	&\leq  \a_k + \frac{\b_k}{\e_k^{1/8}} \leq 2\e_k^{1/8} 
    \end{align}
    
    Now, for each $N\gg 1$, depending on $n$, define the following set
    \begin{align} \label{definition of the set V}
    	V_{N,n} = \bigcap_{k=1}^n G_{N,k}
    \end{align}
    and let $W_n = \limsup_N V_{N,n}$.
    The sets $W_n$ have the following properties:
    \begin{itemize}
    	\item By~\eqref{measure of exceptional set for one N}, since $\nu(G_{N,k}) \geq 1- \e/2^k$, $k=1,\dots,n$, we have $\nu(V_{N,n}) \geq 1 - \e$ for all $N \gg 1$.
              In particular,
                \begin{align} \label{measure of Wn}
                    \nu\left(W_n\right) 
                    =  \nu\left( \limsup_{N\r\infty} V_{N,n} \right) \geq 1 - \e
                \end{align}
         \item For each $y\in W_n$, by~\eqref{estimate of ergodic sum} and~\eqref{definition of the set V}, and noting that $\e>\e_k$, there exists a sequence $A(y,n)\subseteq \N$ defined by
         	\begin{align} \label{the sequence A(y,n)}
         		A(y,n) = \bigcup_{N_i:y\in V_{N_i,n}} [(\e^{1/4}+\e^{1/8})N_i, N_i] \cap \N
         	\end{align}
			such that for all $k = 1,\dots,n$ and for all $Q \in A(y,n)$,
            \begin{align} \label{ergodic sum along A(y,n)}
            	\frac{1}{Q} \sum_{l=1}^{Q} f_k(T_l y) \leq 2\e_k^{1/8} 
            \end{align}
    \end{itemize}

    Let $W = \limsup_n W_n$.
    Then, by~\eqref{measure of Wn},
    \begin{align} \label{measure of W}
    	\nu(W) \geq 1-\e
    \end{align}
    
    Let $y \in W$. Then, there exists a sequence $n_i \r \infty$, such that $y \in W_{n_i}$ for all $i$.
    We will construct a sequence $A(y)$ from the sequences $A(y,n_i)$ defined in~\eqref{the sequence A(y,n)} as follows.
    Let
    \[ \eta = \e^{1/4}+\e^{1/8}  \]
    
    First, we define a sequence $N_i$ by induction. Let $N_0 = 1$. If $N_j$ has been defined, let $N_{j+1}$ be such that the following holds
    \begin{align*}
    	\frac{|A(y,n_{j+1}) \cap [1,N_{j+1}]|}{N_{j+1}} &\geq 1-2\eta\\
        \frac{N_{j}}{N_{j+1}} &\leq \eta
    \end{align*}
    This is possible since the sequences $A(y,n)$ have upper density at least $1-\eta$.
    These conditions imply that
    \begin{align} \label{density of A(y)}
    	\frac{|A(y,n_{j+1}) \cap [N_j,N_{j+1}]|}{N_{j+1}} &\geq 1-3\eta
    \end{align}
    
    Now, define the sequence $A(y)$ as follows:
    \begin{align} \label{definition of A(y)}
    	A(y) = \bigcup_{j=0}^\infty A(y,n_{j+1}) \cap [N_j,N_{j+1}]
    \end{align}
    
    Thus, by~\eqref{density of A(y)}, we get
    \begin{align}
    	\limsup_{N\r\infty} \frac{|A(y)\cap [1,N]|}{N} \geq 1-3\eta
    \end{align}
    
    We claim that
    \begin{align} \label{equidsitribution}
    	\lim_{\substack{N\r\infty\\ N\in A(y)}} \frac{1}{N} \sum_{n=1}^{N} \d_{T_ny} = \mu
    \end{align}
    
    Let $\l_\infty^y$ be any weak-$\ast$ limit of the sequence $\frac{1}{N} \sum_{n=1}^{N} \d_{T_ny}$, $N\in A(y)$.
    First, we claim that $\l_\infty^y(Z) =0$. Suppose otherwise.
    Then, since $Z = \bigcup_i K_i$ and $K_i \subseteq K_{i+1}$ for all $i$, there exists some $i_0$ such that for all $i> i_0$:
    \[\l_\infty^y(K_i) \geq \l_\infty^y(K_{i_0}) > 0\]
    
    Fix some $i>i_0$. By definition of the functions $f_i$, $\l_\infty^y(f_i) \geq \l_\infty^y(K_i)$.
    Then, for all $n_j \geq i$ and for all $N \in A(y,n_{j+1}) \cap [N_j,N_{j+1}] \subset A(y)$, by~\eqref{ergodic sum along A(y,n)}, we get
    \[ \l_\infty^y(f_i) \leq 2\e_i^{1/8} \]
    
    In particular, we get that $\l_\infty^y(K_{i_0}) \leq 2 \e_i^{1/8}$.
    But, this applies to $i>i_0$.
    Thus, since $\e_i \r 0$, we get that $\l_\infty^y(K_{i_0}) =0$, a contradiction.
    
    Next, by our hypothesis, (after possibly intersecting $W$ with a set of full measure), $\l_\infty^y$ is $S$-invariant.
    However, all the ergodic $S$-invariant measures different from $\mu$ live on $Z$ to which $\l_\infty^y$ assigns $0$ mass.
    Thus, by the ergodic decomposition, we get that $\l_\infty^y = \mu$.
    Hence, the sequence $\frac{1}{N} \sum_{n=1}^{N} \d_{T_ny}$, $N\in A(y)$ has $\mu$ as its only weak-$\ast$ limit point as desired.
    
    Thus far, we proved that for all $y\in W$, a set of $\nu$ measure at least $1-\e$, there exists a sequence $A(y)$ of upper density at least $1-3\eta$ such that~\eqref{equidsitribution} holds. 
    Since $\e$ was arbitrary, the conclusion of the theorem holds $\nu$ almost everywhere as desired.

  \end{proof}

\section{Existence of Ratner Sequences}
\label{section unipotent invariance}
    In this section, we prove a general criterion for the existence of Ratner sequences, Theorem~\ref{thrm: existence of Ratner sequences in 1-param unipotent subgroups}.
    Let us fix some notation.
    Let $G$ be a semisimple Lie group and let $H$ be a closed connected Lie subgroup of $G$.
    Let $\G$ be a lattice in $G$ and assume that the orbit $H \G$ is closed in $G/\G$ and supports an $H$-invariant probability measure $\mu_H$.
    Let $Z_G(H)$ denote the centralizer of $H$ in $G$.
    Let $\mf{g}$ and $\mc{h}$ denote the Lie algebras of $G$ and $H$ respectively.
    For $g\in G$, let $Ad(g)$ denote the linear transformation on the Lie algebra of $G$ induced by the adjoint action of $g$.
    Let $\norm{\cdot}$ denote some
    
    Recall that a sequence of elements $g_n \in G$ is said to be a Ratner sequence if there exists a one-parameter unipotent subgroup $W < G$ such that for $\mu_H$-almost every $x \in G/\G$, any limit point of the empirical measures
    $\frac{1}{N}\sum_{n=1}^N \d_{g_n x}$ is invariant by $W$.

\subsection{$SL_2(\R)$ Representations and Unipotent Invariance}
	The first key step in the proof of Theorem~\ref{thrm: existence of Ratner sequences in 1-param unipotent subgroups} is Proposition~\ref{propn: Jacobson-Morozov} below.
    It relies on the representation theory of embedded copies of $SL_2(\R)$ inside semisimple Lie groups and its proof is inspired by Ratner's H-principle appearing in the proof of her measure classification theorem.
    
    The statement roughly says that it is possible to change the starting point of a unipotent orbit of a group $U$ in a direction parallel to $H$ so that the $2$ unipotent orbits differ roughly by a unipotent element in the centralizer of $U$.
    The following is the precise statement.

	\begin{proposition} \label{propn: Jacobson-Morozov}
     Let $U = \set{u_t: t\in \R}$ be an $Ad$-unipotent $1$-parameter subgroup of $G$ such that 
    $U\not\subset Z_G(H)$.
    Then, for every sequence $t_n \r \infty$,
    there exists a sequence $v_n \in \mf{h}=Lie(H)$ satisfying the following: for all $n$ sufficiently large
    \begin{eqnarray*}
    	\norm{v_n} \ll \frac{1}{\norm{u_{t_n}\vert_\mf{h}}}, \qquad
        \exp(Ad(u_{t_n})(v_n)) \xrightarrow{n\r\infty} u
    \end{eqnarray*}
    where $u\in Z_G(U)$ is a non-trivial $Ad$-unipotent element of $G$.
	\end{proposition}
    
    \begin{proof}
    Let $X\in \mf{g}$ be such that $u_t = \exp(tX)$ for all $t\in \R$.
    Since $u_t$ is $Ad$-unipotent, $X$ is an $ad$-nilpotent element of $\mf{g}$.
    That, $ad(X)$ is a nilpotent linear transformation of $\mf{g}$.
    
    Since $\mf{g}$ is a semisimple Lie algebra, by the Jacobson-Morozov theorem, $X$ may be extended to an $\mf{sl}_2$-triple.
    That is there exists $Y, H \in \mf{g}$ such that the following relations hold:
    \begin{align*}
    	[H,X] = 2X, \qquad
        [H,Y] = -2Y, \qquad
        [X,Y] = H
    \end{align*}
    Hence, the Lie subalgebra $\mf{f}$ generated by $X,Y$ and $H$ is isomorphic to the Lie algebra $\mf{sl}_2(\R)$.
    Since $\mf{f}$ is semisimple, $\mf{g}$ decomposes into irreducible representations under the adjoint action of $\mf{f}$ as follows:
    \[ \mf{g} = V_1 \oplus \cdots \oplus V_s \]
    For $1\leq i \leq s$, let $\pi_i: \mf{g} \r V_i $ denote the associated projections and note that $Ad(u_t)$ commutes with $\pi_i$ for all $i$.
    Let $v\in \mf{h}$.
    
    Let $1\leq i\leq s$ be such that $\pi_i(v)$ is not fixed by $Ad(u_t)$.
    Let $n_i \in \N$ be such that
    \[ \mathrm{dim}(V_i) = n_i + 1 \]
    By the standard description of irreducible $\mf{sl}_2(\R)$ representations, $V_i$ decomposes into $1$ dimensional eigenspaces for the action $H$ as follows:
    \[ V_i = W_0^{(i)} \oplus W_1^{(i)} \oplus \cdots \oplus W_{n_i}^{(i)}  \]
    so that for each $0\leq l \leq n_i$ and every $w \in W_l^{(i)}$, we have
    \[ H \cdot w = (n_i-2l) w \]
    Let $q_l: V_i \r W_l^{(i)}$ denote the associated projections.
    Let $\set{w_l^{(i)}:0\leq l\leq n_i}$ denote a basis of $V_i$ consisting of eigenvectors of $H$ and write
    \[ \pi_i(v) =  \sum_{l=0}^{n_i} c_{l}^{(i)} w_l^{(i)} \]
    Note that for each $l$, we have that
    \begin{equation*}
    	Ad(u_t)\cdot w_l^{(i)} =\sum_{k=0}^l \binom{l}{k} t^{l-k}w_k^{(i)}
    \end{equation*}
    In particular, we get the following
    \begin{equation} \label{eqn:highest weight coefficient}
     q_0 ( \pi_i(Ad(u_t)(v))) = q_0 ( Ad(u_t)(\pi_i(v)) ) = 
     	 \sum_{k=0}^{n_i}	c_k^{(i)}  t^{k} w_{0}^{(i)}
    \end{equation}
    Note also that the degree of the polynomial appearing in the coefficient of $q_{l} ( \pi_i(Ad(u_t)(v)))$ for any $l>0$ is strictly less than the degree of the polynomial in~\eqref{eqn:highest weight coefficient}.
    Let
    \[ d_i(v) = \max\set{0\leq k\leq n_i: c_k^{(i)} \neq 0} \]
    And, define the following natural number
    \[ d_\mf{h} = \max\set{d_i(v): 1\leq i \leq s, v\in \mf{h}} \]

    By assumption, we have that $U \not\subset Z_G(H)$.
    Thus, $d_\mf{h} \neq 0$.    
    In particular, we can find $ v\in \mf{h}$ and $1\leq i\leq s$ such that 
    \[ d_i(v) = d_\mf{h}  \]
    
    Now, let $t_n \in \R\setminus\set{0}$ be a sequence tending to infinity and define $v_n$ to be
    \[ v_n = \frac{1}{t_n^{d_\mf{h}} } v \]
    Thus, for each $i$ with $d_i(v) = d_\mf{h}$, by~\eqref{eqn:highest weight coefficient}, we get that
    \[ q_0 ( \pi_i(Ad(u_{t_n})(v_n))) = (c_{d_\mf{h}}^{(i)} + O(1/t_n)) w_{0}^{(i)}  \]
    And, for each $l >0$, we have that
    \[q_{l} ( \pi_i(Ad(u_{t_n})(v_n))) \xrightarrow{n\r\infty} 0 \]
    
    In particular, we get that
    \begin{equation} \label{eqn: limit vector}
    Ad(u_{t_n})(v_n) \xrightarrow{n\r\infty} 
    	\sum_{\substack{1\leq i\leq s\\ d_i(v) = d_\mf{h} }} c_{d_\mf{h}}^{(i)}  w_{0}^{(i)}  \neq 0
    \end{equation} 
    
    Next, since $v_n \r 0$, $\exp(v_n)$ converges to the identity element in $G$.
    Hence, all the eigenvalues of the linear transformation $Ad(\exp(v_n))$ tend to $1$.
    Thus, since conjugation doesn't change eigenvalues, we get that 
    $Ad(u_{t_n})(\exp(v_n))$ converges to an $Ad$-unipotent element of $G$ which is non-trivial by~\eqref{eqn: limit vector}.
    Since $Ad(u_t)$ fixes $w_0^{(i)}$ for all $i$, by~\eqref{eqn: limit vector}, the limiting element belongs to the centralizer of $U$.
    
    Finally, note that equation~\eqref{eqn:highest weight coefficient} and the definition of $d_\mf{h}$ imply that for all $t>0$ sufficiently large
    \begin{equation} \label{eqn: norm of u_t is polynomial in t}
     \norm{Ad(u_{t})\vert_\mf{h} } = O(t^{d_\mf{h}})
    \end{equation}
    This completes the proof.
    \end{proof}
    
    The following example demonstrates the above Proposition in the concrete set up of $G = SL_2(\R)$.
    \begin{example} \label{sl(2,R) example} Let $G = SL(2,\R)$ and $H=K = SO(2)$. Let $t_n \r +\infty$ be a sequence.
  Let $g_n$ be the following sequence:
  \[ g_n = \begin{pmatrix}
            1 & t_n \\ 0 & 1
            \end{pmatrix}
  \]
  
  Let $k_\theta \in K$. Then, 
  \[
  	g_n k_\theta g_n^{-1} =
    	\begin{pmatrix}
            \cos(\theta) -t_n\sin(\theta) & (t^2_n+1)\sin(\theta) \\ 
            - \sin(\theta) & \cos(\theta) +t_n\sin(\theta)
        \end{pmatrix}
  \]
  
  Let $\a \in (0,1)$ be a fixed real number. For all large $n$, let $\theta_n$ be such that $t_n^2 \sin(\theta_n) = \a$.
  Then, as $n\r\infty$, $\theta_n \r 0$ and $t_n\sin(\theta_n) \r 0$.
  Hence, we get
  \[ g_n k_{\theta_n} g_n^{-1} \r u(\a) = 
  	\begin{pmatrix}
        1 & \a \\ 0 & 1
     \end{pmatrix} \neq id
  \]
  \end{example}

\subsection{Decay of Correlations}
\label{section: decay of correlations}

	Let $\G$, $H$ and $G$ be as above and define
    \[ g_n = u_{t_n} \]
    Let $\vp \in C_c^\infty(G/\G)$. 
    For each $n$, define the following function on $H\G/\G$:
    \begin{align} \label{definition of f_n}
      f_n(h\G) = \vp(g_n\exp(v_n)h\G) - \vp(g_nh\G)
    \end{align} 
    where $v_n\in Lie(H)$ is as in the conclusion of Proposition~\ref{propn: Jacobson-Morozov} applied to the sequence $t_n$.
    The reason for defining such functions is the following
  
  \begin{proposition}
  \label{f_n go to zero implies unipotent invariance}
  	To prove Theorem~\ref{thrm: existence of Ratner sequences in 1-param unipotent subgroups}, it suffices to show that for $\mu_H$ almost every $x\in H\G/\G$, the following holds:
    \[ \frac{1}{N} \sum_{n=1}^N f_n(x) \r 0\]
  \end{proposition}
  
  \begin{proof}
  	Let $h_n = \exp(v_n)$.
  	By Proposition~\ref{propn: Jacobson-Morozov}, we have $g_n h_ng_n^{-1} \r u$, where $u$ is a non-trivial Ad-unipotent element.
    Since, $\vp$ is uniformly Lipschitz, we have
    \[ \vp(ug_nx) - \vp(g_nh_nx) = \vp(ug_nx) - \vp(g_nh_ng_n^{-1}g_nx)
    = O\left(d(u, g_nh_ng_n^{-1})\right) \]
    where $d(.,.)$ is the right invariant metric on $G$.
    Hence, $|\vp(ug_nx) - \vp(g_nh_nx)| \r 0$ as $n\r\infty$.
    
    Hence, by assumption,
    \begin{align*}
    	\left| \frac{1}{N}\sum_{n=1}^N ( \vp(ug_nx)-\vp(g_nx) ) \right|
        	&\leq \frac{1}{N}\sum_{n=1}^N \left| \vp(ug_nx)-\vp(g_nh_nx) \right| 
            + \left| \frac{1}{N}\sum_{n=1}^N f_n(x) \right| \\
            &\longrightarrow 0
    \end{align*}
    
    But, $\vp$ was an arbitrary function. Thus, any limit point must be invariant by the group generated by $u$.
    	
  \end{proof}
  
  In order to use Proposition~\ref{f_n go to zero implies unipotent invariance}, we shall need the following estimate on the correlation between the functions $f_n$.
  This lemma is an adaptation of the key technical lemma in~\cite{EskinChaika} (Lemma 3.3) to our setting. we shall need the following definition.
  
  \begin{definition} \label{definition of injectivity radius}
  	For any $x\in H\G/\G \cong H/(H\cap\G)$, the injectivity radius at $x$, denoted by $inj_x$ is defined to be the infimum over all $r>0$ such that the map $h\mapsto hx$ is injective on the ball of radius $r$ around identity in $H$.
  \end{definition}

  \begin{lemma}\label{decay of correlations}
  	For all $n\geq m \geq 1$ such that $\norm{Ad(g_m)\vert_\mf{h}}/ \norm{Ad(g_n)\vert_\mf{h}}$ is sufficiently small, the following holds
    \begin{align}
    	\int f_n(h\G)f_m(h\G) \;d\mu_H(h\G) 
        = O\left( \left(\frac{\norm{Ad(g_m)\vert_\mf{h}}}{\norm{Ad(g_n)\vert_\mf{h}}}\right)^{1/2} \right)
    \end{align}
  \end{lemma}
  
  \begin{proof}
  	Let $v_n$ be as in the definition of the functions $f_n$ and let $h_n = \exp(v_n)$.
    Let $d_n = \norm{Ad(g_n)\vert_\mf{h}}$ and $d_m = \norm{Ad(g_n)\vert_\mf{h}}$. Define
    \begin{align} \label{definition of r}
    	r = \left(\frac{1}{d_md_n}\right)^{1/2}
    \end{align}
    
    Let $B_H(e,r)$ denote the ball of radius $r$ around the identity in $H$. By abuse of notation, we'll use $\mu_H$ to denote the Haar measure on $H$ and on $H\G/\G$.
    
    Let $\psi:H\G/\G \r \R$ be any integrable function.
    Then, by Fubini's theorem and left $H$-invariance of $\mu_H$,
    \begin{align*}
        \int_{H\G/\G} \psi(x)\;d\mu_H(x) =
        	\int_{H\G/\G} \frac{1}{\mu_H(B_H(e,r))} \int_{B_H(e,r)} \psi(hx) \;d\mu_H(h) \;d\mu_H(x)
    \end{align*}
    
    Define the following set
    \[ Thick_r = \set{x\in H\G/\G: inj_x \geq r} \]
    where $inj_x$ denotes the injectivity radius at $x$ in $H/(H\cap\G)$. Let $Thin_r = H\G/\G - Thick_r$.
    Using the structure of Siegel sets, one can show (Lemma 11.2,~\cite{BenoistOh}) that $\mu_H(Thin_r) \ll r^p$, for some $p>0$ as $r\r 0$.
    Hence, it suffices to prove for all $x\in Thick_r$,
    \begin{align} \label{estimate for average on balls}
    	\frac{1}{\mu_H(B_H(e,r))} \int_{B_H(e,r)} f_n(hx)f_m(hx) \;d\mu_H(h) = 
        O\left( \left(\frac{\norm{Ad(g_m)\vert_\mf{h}}}{\norm{Ad(g_n)\vert_\mf{h}}}\right)^{1/2} \right)
    \end{align}

    Let $w\in Thick_r$ be fixed.    
    Let $B_r$ denote the ball of radius $r$ around the $w$ in $H\G/\G$ in the metric induced by the metric on $G$.
    Then, for every $x \in B_r$, there exists some $l\in B_H(e,r)$ such that $x = lw$.    
    Since $\vp \in C_c^\infty(G/\G)$, $\vp$ is uniformly Lipschitz. Thus, we get
    \begin{align*}
    	\vp(g_mh_mx) - \vp(g_mh_mw) = \vp(g_mh_mlw)-\vp(g_mh_mw)
        	= O\left( d(g_mh_mlh_m^{-1}g_m^{-1},e)  \right)   
    \end{align*}
    	
	Since the sequences $d_n, d_m$ are tending to infinity, for all $n,m$ sufficiently large, $r$ will be small enough so that the exponential map is a diffeomorphism from a neighborhood of $0$ in $\mf{h} = Lie(H)$ onto $B_H(e,r)$.
    
    Thus, we can write $l=\exp(v)$ for some $v\in \mf{h}$.  So, we have
    \[ \norm{Ad(g_mh_m)(v)} \leq \norm{Ad(g_m)\vert_\mf{h}} \cdot \norm{Ad(h_m)} \cdot \norm{v} \]
    where for any $g\in G$.
    
    But, since $h_m \r id$ as $m\r\infty$ and since the norm is continuous, for all $m$ sufficiently large, we have $\norm{Ad(h_m)} \ll 1$.
    
    Moreover, since the differential of the exponential map at $0$ is the identity, its Jacobian is $1$ at $0$ and hence, when $r$ is sufficiently small, we have $\norm{v} \ll d(l,e) \leq r$.
    Combining these estimates, we get for all $x\in B_r$,
    \begin{align*} 
    	\norm{Ad(g_mh_m)(v)} = O\left( \norm{Ad(g_m)\vert_\mf{h}} r \right) = O\left( \left(\frac{\norm{Ad(g_m)\vert_\mf{h}}}{\norm{Ad(g_n)\vert_\mf{h}}} \right)^{1/2}\right)
    \end{align*}
    
    But, as before, the exponential map is nearly an isometry near identity. Hence, when $\norm{Ad(g_n)\vert_\mf{h}}$ is sufficiently larger than $\norm{Ad(g_m)\vert_\mf{h}}$, $Ad(g_mh_m)(v)$ will be sufficiently close to $0$ so that $d(\exp(Ad(g_mh_m)(v)),e)\ll \norm{Ad(g_mh_m)(v)}$ up to absolute constants.
    Thus, we get for all $x\in B_r$,
    \[ \vp(g_mh_mx) - \vp(g_mh_mw) = O\left( \left(\frac{\norm{Ad(g_m)\vert_\mf{h}}}{\norm{Ad(g_n)\vert_\mf{h}}} \right)^{1/2}\right)\]
    
    Similarly, we get the same estimate for $\vp(g_mx) - \vp(g_mw)$ for all $x\in B_r$.
    Thus, by definition of $f_m$, we get
    \[ 
     	f_m(x) - f_m(w) = O\left( \norm{Ad(g_m)\vert_\mf{h}} r \right)
     \]
    
    Thus, we get that
    \begin{align} \label{estimate f_m away}
        \frac{1}{\mu_H(B_r)} \int_{B_r} f_n(x)f_m(x)\;d\mu_H(x) 
        = \frac{f_m(w)}{\mu_H(B_r)} \int_{B_r} f_n(x)\;d\mu_H(x)
        + O\left( \norm{Ad(g_m)\vert_\mf{h}} r \right)
    \end{align}
    
    Next, note that by definition of $f_n$ and left-invariance of $\mu_H$,
    \begin{align*}
    	\frac{1}{\mu_H(B_r)}\int_{B_r} f_n(x)\;d\mu_H(x) 
        	&= \frac{1}{\mu_H(B_r)}\int_{h_nB_r} \vp(g_nx)\;d\mu_H(x) 
        				-\frac{1}{\mu_H(B_r)}\int_{B_r} \vp(g_nx)\;d\mu_H(x)\\
                        &=  O\left( \frac{\mu_H( h_nB_r\triangle B_r)}{\mu_H(B_r)}  \right)
    \end{align*} 
    	
   Since $w\in Thick_r$, $B_r$ isometric to $B_H(e,r)$. 
   Hence, for all $r$ sufficiently small, we may apply Proposition~\ref{measure of symmetric difference} below to get
   \begin{align} \label{estimate f_n}
   		\frac{1}{\mu_H(B_r)}\int_{B_r} f_n(x)\;d\mu_H(x)  = O\left( \frac{d(h_n,e)}{r} \right)
   \end{align}  
   
   Note that Proposition~\ref{measure of symmetric difference} requires that $h_n \in B_H(e,r)$.
   To get around this assumption, observe that for $n$ sufficiently large, $h_n =\exp(v_n)$ will be sufficiently close to identity and hence, we have
   		\[ d(h_n,e) \ll \norm{v_n} \]
   
   But, by Proposition~\ref{propn: Jacobson-Morozov}, we have $\norm{v_n} \ll 1/\norm{Ad(g_n)\vert_\mf{h}}$ for all $n$ sufficiently large.
   But, since $\norm{Ad(g_n)\vert_\mf{h}} \geq \norm{Ad(g_m)\vert_\mf{h}}$ by assumption, we have $\norm{Ad(g_n)\vert_\mf{h}} \geq 1/r$.
   Thus, in particular, $h_n$ will be contained in a ball of radius comparable to $r$ for all large $n$, which doesn't affect our estimate.
   
   Moreover, this observation, along with~\eqref{estimate f_n}, imply that 
   \begin{align} 
   		\frac{1}{\mu_H(B_r)}\int_{B_r} f_n(x)\;d\mu_H(x)  = O\left( \frac{1}{r\norm{Ad(g_n)\vert_\mf{h}}} \right)
   \end{align}
   
   Combining this estimate with~\eqref{estimate for average on balls} and~\eqref{estimate f_m away} gives
   \begin{align}
   		\int_{Thick_r} f_n(h\G)f_m(h\G)\;d\mu_H(h\G) = O\left( \left( \frac{\norm{Ad(g_m)\vert_\mf{h}}}{\norm{Ad(g_n)\vert_\mf{h}}} \right)^{1/2} \right)
   \end{align}
   and the conclusion of the lemma follows.
  \end{proof}

  \subsubsection{A measure estimate}
  
  The following estimate was used in the proof of Lemma~\ref{decay of correlations}.
  \begin{proposition} \label{measure of symmetric difference}
  	Let $H$ be a Lie group and let $B_r$ denote a ball of radius $r>0$ around the identity in $H$.
    Then, for all $r>0$ sufficiently small and all $h\in B_r$,
  	\[\frac{\mu_H( hB_r\triangle B_r)}{\mu_H(B_r)} = O\left( \frac{d(h,e)}{r} \right) \]
  	where $\mu_H$ denotes a left-invariant Haar measure on $H$ and $d(.,.)$ denotes a right invariant metric.
  \end{proposition}
  
  \begin{proof}
  	Let $\mf{h} = Lie(H)$. Fix a norm on $\mf{h}$ inducing the metric $d$.
    Let $r>0$ be small enough such that the exponential map is a diffeomorphism from a ball around $0$ in $\mf{h}$ onto $B_r$.
    Since the differential of the exponential is the identity at $0$, such ball will have a radius comparable to $r$, denote it by $B_{r'}^{\mf{h}}$
    
    Let $g\in B_r$. Let $X,Y\in \mf{h}$ be such that $h = \exp(X)$ and $g = \exp(Y)$.
    Then, if $r$ is sufficiently small, by the Campell-Baker-Hausdorff formula, there exists some $Z\in \mf{h}$ so that $hg = \exp(Z)$ and
    \[ Z-Y = X + o(\norm{X})  \]
    
    In particular, there is some $C\geq1$ such that 
    $hB_r\subseteq \exp(B_{r'}^{\mf{h}} + CX) $.
    And, hence, we get that
    \begin{align*}
    	hB_r\triangle B_r \subseteq \exp((B_{r'}^{\mf{h}} + CX) \triangle B_{r'}^{\mf{h}})	
    \end{align*}

    Let $Leb$ denote the Lebesgue measure on $\mf{h}$. It is then a standard fact from convex euclidean geometry that
    \begin{align} \label{estimate of measure of symmetric difference}
    	Leb((B_{r'}^{\mf{h}} + CX) \triangle B_{r'}^{\mf{h}}) \ll \norm{X} r^{dim H -1}
    \end{align} 
    where the implicit constants are absolute and depend only on the dimension (see for example~\cite{GroemerConvexGeometry}).
    Here we are using that a ball in the norm on $\mf{h}$ is equivalent to a standard euclidean ball of comparable radius.
    
    Again, since the differential of the exponential is the identity at $0$, the Haar measure on on $H$ near identity is comparable up to absolute constants with the pushforward of the Lebesgue measure under the exponential map.
    
    In particular, one has $\mu_H(B_r) \ll r^{dim H}$. Combining this with~\eqref{estimate of  measure of symmetric difference} gives the desired conclusion.
  \end{proof}

\subsection{Law of Large Numbers}
\label{section: LLN}

This section is dedicated to the proof of Proposition~\ref{law of large numbers proposition} below. The proof follows the same ideas as in the proof of the strong law of large numbers and is very similar to the proof in~\cite{EskinChaika}. The main difference is the need to handle the general growth condition in Theorem~\ref{thrm: existence of Ratner sequences in 1-param unipotent subgroups}.

Recall the definition of the functions $f_n$ from the previous section.
By Proposition~\ref{f_n go to zero implies unipotent invariance}, the following proposition concludes the proof of Theorem~\ref{thrm: existence of Ratner sequences in 1-param unipotent subgroups}.

  \begin{proposition} \label{law of large numbers proposition}
	Under the same hypotheses of Theorem~\ref{thrm: existence of Ratner sequences in 1-param unipotent subgroups}, for $\mu_H$ almost every $x\in H\G/\G$,
    \[ \frac{1}{N} \sum_{n=1}^N f_n(x) \r 0 \]
  \end{proposition}

  \begin{proof} 
  For $x\in H\G/\G$ and $N \in \N$, let $S_N(f)(x) =\sum_{n=1}^N f_n(x)$.
  As noted in equation~\eqref{eqn: norm of u_t is polynomial in t}, there exists a natural number $d \geq 1$, depending only on $H$ and $U$, such that
  \[ \norm{Ad(g_n)\vert_\mf{h}} = \norm{Ad(u_{t_n})\vert_\mf{h}} = O(\left(t_n^d\right)   \]
  as $n$ tends to infinity.
  Hence, by assumption, there exists $\l>0$ such that for all $n\geq 1$, we have
  \begin{align} \label{growth rate}
  	\frac{\norm{Ad(g_m)\vert_\mf{h}}}{\norm{Ad(g_n)\vert_\mf{h}}} \ll e^{d \l(m- n )} 
  \end{align}

  Then, we have
  \begin{align*}
    	\int |S_N(f)(x)|^2 \;d\mu_H(x) &= \sum_{1\leq n,m\leq N}\int  f_n(x)f_m(x) \;d\mu_H(x) \\
       	&= O\left(N^{3/2}\right) + 
           	\sum_{|n-m| \geq N^{1/2}} \int  f_n(x)f_m(x) \;d\mu_H(x) \\ 
   \end{align*}

   Here we estimated the number of pairs $(m,n)$ with $|m-n|< N^{1/2}$ using the area between the $2$ lines $m \pm n = N^{1/2}$ in the square $[0,N]^2$.
   
   But, by~\eqref{growth rate}, when $N\gg 1$, for $n \geq m$ such that $|n-m| \geq N^{1/2}$, we have that $\norm{Ad(g_m)\vert_\mf{h}}/\norm{Ad(g_n)\vert_\mf{h}}$ will be sufficiently small so that Lemma~\ref{decay of correlations} applies.
   This implies that for all $N\gg 1$:
   \begin{align*}
   		\frac{1}{N^2}\int |S_N(f)(x)|^2 \;d\mu_H(x) 
        	&= O\left(N^{-1/2}\right) + O \left( e^{\frac{-d\l N^{1/2}}{2}} \right) \\
            &= O\left(N^{-1/2}\right)
   \end{align*}
   
   Let $\e>0$. Then, by the Chebyshev-Markov inequality,
   \begin{align*}
   		\mu_H \left(\set{x: \left|\frac{S_N(f)(x)}{N}\right|>\e}   \right)
         \ll \frac{N^{-1/2}}{\e^2}
   \end{align*}

   For all $k \in \N$, let $N_k = k^{4} $.
   Thus, the above observation shows that the sequence  $N_k^{-1/2}$ is summable.
   Hence, by the Borel-Cantelli lemma, we have
   \begin{align*}
   		\mu_H\left( x: \left|\frac{S_{N_k}(f)(x)}{N_k}\right|>\e \text{ for infinitely many } k  \right) = 0
   \end{align*}
   
   Since $\e$ was arbitrary, by taking a countable sequence $\e_i$ decreasing to $0$, we conclude that for $\mu_H$ almost every $x$, 
   \[ \lim_{k\r\infty} \frac{S_{N_k}(f)(x)}{N_k} = 0 \]
   
   We are left with bootstrapping this conclusion to all sequences, for which we use a standard interpolation argument.
   Let $M_i \r \infty$ be a sequence.
   Observe that for each $M_i \in \N$, there exists some $k_i\in \N$ such that $N_{k_i} \leq M_i \leq N_{k_i+1}$.
   
   Moreover, we have $N_{k_i+1} - N_{k_i} = O(k_i^{3})$.
   Thus, we get that
   \begin{align*}
   		\left| \frac{S_{M_i}(f)(x)}{M_i}  \right|
        &\leq \frac{N_{k_i}}{M_i} \left| \frac{S_{N_{k_i}}(f)(x)}{N_{k_i}} \right| + O(k_i^{-1})
        \xrightarrow{i\r\infty} 0
   \end{align*}
   as desired.
	\end{proof}

\section{Proof of Theorem~\ref{Intro-Pointwise equidistribution for general groups}}
\label{section-general groups}
	
    This section is dedicated to the proof of Theorem~\ref{Intro-Pointwise equidistribution for general groups}.
	Let the notation be the same as in \S~\ref{section unipotent invariance}.
    Recall that a sequence $g_n$ of elements of $G$ is said to be a \textit{Ratner Sequence} if there exists a non-trivial $Ad$-unipotent element
      $u \in G$ such that for $\mu_H$ almost every $x\in G/\G$,
      any limit point of the sequence of empirical measures $\frac{1}{N} \sum_{n=1}^N \d_{g_n x}$ is invariant by $u$.
    
    The existence of such sequences is the content of Theorem~\ref{thrm: existence of Ratner sequences in 1-param unipotent subgroups}.

    \begin{proof}[Proof of Theorem~\ref{Intro-Pointwise equidistribution for general groups}]

    Let $g_n$ be a Ratner sequence for $H$. 
    Let $W$ denote the one-parameter unipotent subgroup generated by an $Ad$-unipotent element $u$ as in the definition of Ratner sequences above.
    
    We will apply Theorem~\ref{Intro-Birkhoff for sequences of transformations} with $X = G/\G$, $\mu = \mu_{G/\G}$, $\nu = \mu_H$ and $T_n=g_n$.
    Let us verify the hypotheses.
    By assumption, we have that
    \[ \frac{1}{N}\sum_{n=1}^N (g_n)_\ast \nu \r \mu \]
    In particular, this implies condition $(1)$ of Theorem~\ref{Intro-Birkhoff for sequences of transformations}.
    Moreover, by definition of Ratner sequences, condition $(2)$ is satisfied, with the transformation $S$ being multiplication by the unipotent element $u$.
    
    Let $\mc{L}$ denote the collection of proper analytic subgroups $L$ of $G$ such that $L\cap \G$ is a lattice.
    Then, $\mc{L}$ is a countable set~\cite{RatnerMeasure}.
    
    For $L\in \mc{L}$, define the following set
    \begin{align*}
    	N(L,W) = \set{g\in G: g^{-1}Wg\subseteq L }
    \end{align*}
    
    Let $\pi: G \r G/\G$ denote the natural projection.
    The set $Z$ appearing in the hypotheses of Theorem~\ref{Intro-Birkhoff for sequences of transformations} will be defined to be
    \begin{align*}
    	Z = \bigcup_{L\in \mc{L}} \pi\left(N(L,W) \right)
    \end{align*}
    
    Then, since $Z$ is a countable union of analytic subvarieties of $G/\G$~\cite{RatnerMeasure}, $Z$ admits a filtration by compact sets.
    Moreover, since $\mc{L}$ is countable, and $\mu_{G/\G}(\pi(N(L,W))) = 0$ for all $L\in \mc{L}$, we have
    \[ \mu_{G/\G}(Z) = 0 \]
    
    Finally, by Ratner's measure rigidity theorem~\cite{RatnerMeasure}, any ergodic $W$ invariant probability measure $\l \neq \mu_{G/\G}$ is supported on $N(L,W)$ (in fact supported on a single closed orbit of a conjugate of $L$) for some $L\in\mc{L}$.
    Thus, all the hypotheses of Theorem~\ref{Intro-Birkhoff for sequences of transformations} are verified and hence the conclusion of Theorem~\ref{Intro-Pointwise equidistribution for general groups} follows.
    
    \end{proof}

\appendix

\section{Symmetric Groups and Ratner Sequences}

	Throughout this section, $G$ is a connected semisimple Lie group with finite center and $H$ is symmetric subgroup of $G$.
	We let $\G$ be an irreducible lattice in $G$ and assume the orbit $H \G$ is closed in $G/\G$ and supports an $H$-invariant probability measure.
    
	We prove a general criterion for a sequence of elements of a Lie group $G$ to contain a Ratner sequence as a subsequence with respect to $H$.
    The precise statement is Theorem~\ref{unipotent invariance} below.    
    The proof of this theorem follows the same lines as the proof of Theorem~\ref{thrm: existence of Ratner sequences in 1-param unipotent subgroups}.
    The only difference being Proposition~\ref{propn: Mautner phenomenon} below which acts as a replacement for Proposition~\ref{propn: Jacobson-Morozov} in the proof of Theorem~\ref{thrm: existence of Ratner sequences in 1-param unipotent subgroups}.
    The arguments in sections~\ref{section: decay of correlations} and~\ref{section: LLN} carry over verbatim to this setting.

\subsection{Structure of Affine Symmetric Spaces}
  Our main tool will be the structure of affine symmetric spaces which we recall here.
  We follow the exposition in~\cite{eskin1993} closely for the material in this section.
  Let $\s: G \r G$ be an involution such that $H$ is the fixed point set of $\s$. 
  Then, $G/H$ is an affine symmetric space.
  By abuse of notation, let $\s$ also denote the differential of $\s$ at identity.
  Let $\mf{g}$ denote the Lie algebra of $G$.
    Then, we have
    \[ \mf{g} = \mf{h} \oplus \mf{p} \]
    where $\mf{h}$ is the eigenspace corresponding to the eigenvalue $1$ of $\s$ and $\mf{p}$ corresponds to the $-1$ eigenspace,
    and $\mf{h}$ is the Lie algebra of $H$.
    
    It is well known (Proposition 7.1.1,~\cite{Schlichtkrull}) that one can find a Cartan involution $\theta$ of $G$ commuting with $\s$. Let $\theta$ also denote its differential at identity.
    Then, similarly $\mf{g}$ splits as
    \[ \mf{g} = \mf{k}\oplus \mf{q} \]
    where $\mf{k}$ (resp. $\mf{q}$) is the $+1$ (resp. $-1$) eigenspace of $\theta$.
    Since $\theta$ is a Cartan involution, $\mf{k}$ is the Lie algebra of a maximal compact subgroup, denote it by $K$.
    
    Now, let $\mf{a}$ be a maximal abelian subspace of $\mf{p}\cap \mf{q}$.
    Then, $\mf{a}$ is the Lie algebra of a maximal abelian subgroup $A$ and the exponential map $\mf{a}\r A$ is a diffeomorphism.

  	Recall that $G$ admits a decomposition of the form $G = KAH$ (See~\cite{Schlichtkrull}, Proposition 7.1.3 or~\cite{eskin1993}, Proposition 4.2).
    Elements of the fiber of the map $ (k,a,h)\mapsto kah$ have the form $(kl,a,l^{-1}h)$ for some element $l\in K\cap H$.
    In particular, the fiber lies in a compact group.

    Consider the adjoint action of $\mf{a}$ on $\mf{g}$.
    There exists a finite subset $\Sigma \subset{\mf{a}^\ast}$ of non-zero elements of the dual of $\mf{a}$ such that $\mf{g}$ splits as
    \[ \mf{g}  = \mf{g}_0 \oplus \bigoplus_{\a \in \Sigma} \mf{g}_\a \]
    such that for all $X\in \mf{a}$ and all $Z\in \mf{g}_\a$,
    \[ ad_{X}(Z) = \a(X)Z \]
    And, for $Z\in \mf{g}_0$, $ad_X(Z) = 0$.
    Recall that the subspaces
    \[ \set{X\in \mf{a}: \a(X) = 0} \]
    for $\a\in \Sigma$ divide $\mf{a}$ into a finite collection of cones, called Weyl Chambers.

   Let $\mf{C}$ be one such Weyl chamber.
   Let $\Sigma^+$ denote the set of $\a\in \Sigma$ such that
   $\a(X) > 0$ for all $X\in \mf{C}$. We call $\Sigma^+$ the set of positive roots relative to $\mf{C}$. Then, $\mf{g}$ splits as follows:
   \[ \mf{g} = \mf{n}^- \oplus \mf{g}_0 \oplus \mf{n}^+ \]
   
   where $\mf{n}^- = \bigoplus_{\a \in \Sigma-\Sigma^+} \mf{g}_\a$ and
   $\mf{n}^+ = \bigoplus_{\a \in \Sigma^+} \mf{g}_\a$.
    
\subsection{Unipotent Invariance}

      We need to fix some notation before stating the main theorem.
      Fix some norm on $Lie(G)$ inducing the metric on $G$ and denote it by $\norm{\cdot}$.
      This norm induces a matrix norm for the Adjoint maps.
      Let $\norm{Ad(g)}$ denote such matrix norm.

      The following is the main theorem of this section.

\begin{theorem} \label{unipotent invariance}
	In the notation above, if $g_n\in G$ is a sequence tending to infinity in $G/H$ and satisfying the following growth condition:
    \begin{enumerate}
    	\item There exists a constant $\l > 0$ such that for all $ n \geq 1$,
        	\[  \norm{Ad(g_n)}=O\left(e^{\l n}\right)  \]
        \item Writing $g_n = k_na_nh_n$, we have that $\norm{Ad(h_n^{-1})}$ is uniformly bounded for all $n$.
    \end{enumerate} 
    Then, $g_n$ contains a Ratner sequence for $H$ as a subsequence.

\end{theorem}

	\begin{remark}
    	\begin{enumerate}
        	\item Passage to a subsequence in the conclusion of Theorem~\ref{unipotent invariance} is needed to insure invariance in the limit by a \textit{single} one-parameter unipotent subgroup.
      This is very important for applying Theorem~\ref{Intro-Birkhoff for sequences of transformations} to prove Theorem~\ref{Intro-Pointwise equidistribution for general groups}.
    		\item The second growth condition in Theorem~\ref{unipotent invariance} makes sense, since the element $h_n$ in the decomposition of $g_n$ is unique up to left multiplication by elements inside the compact group $H\cap K$.
        
      		\item The growth rate of $\norm{Ad(g_n)}$ required by this theorem is not the most general one which works with our techniques.
        It is possible to obtain the same conclusions assuming there exist constants $\l,c >0$ such that $ \norm{Ad(g_n)} =O\left( e^{\l n^c}\right)$.
        \end{enumerate}
	\end{remark}    

\subsection{Expansion Properties of the Adjoint Action}

	We shall need the following lemma regarding the Adjoint action of $G$.
    This lemma exploits the relationship between diagonalizable elements and their associated horospherical subgroups.
    We also make use of the structure of affine symmetric spaces.

  \begin{proposition} \label{propn: Mautner phenomenon}
	Let $g_n$ be as in Theorem~\ref{unipotent invariance}.
    Then, there exists a sequence $v_n \r 0 \in Lie(H)$ satisfying the following for all $n$,
    \[ \norm{v_n} \ll \frac{1}{\norm{Ad(g_n)}} \]
    and such that after passing to a subsequence of the $g_n$'s, we have
    \[ g_n \exp(v_n)g_n^{-1} \r u \neq id \]
    where $u$ is an Ad-unipotent element in $G$.
  \end{proposition}

   We will need the following fact for the proof of Proposition~\ref{propn: Mautner phenomenon}.
  
  \begin{lemma} (Lemma 3,~\cite{Mozes1992}) \label{adjoint action is proper}
  	If $G$ is semisimple over $\R$ with finite center, then the Adjoint representation $Ad: G \r GL(\mf{g})$ is a proper map.
  \end{lemma}
  
  We are now ready for the proof.
  \subsubsection{Proof of Proposition~\ref{propn: Mautner phenomenon}}

     Write $g_n = k_na_nh_n$. Then, by passing to a subsequence of $g_n$, we may assume that there exists a single Weyl chamber $\mf{C}$ such that $a_n = \exp(X_n)$ and $X_n \in \mf{C}$ for all $n$.
     Let $\Sigma^+$ be a set of positive roots associated with $\mf{C}$.
     
     First, we'll assume that $g_n \in KA$ and write $g_n = k_na_n$.
     Note that the map $Ad: G \r GL(\mf{g})$ is proper by Lemma~\ref{adjoint action is proper}.
     In particular, by assumption, since $g_n \r \infty$, we have
     \[ \norm{Ad(g_n)} \r \infty  \]

     We claim that the image of $\mf{h}$ under the projection onto $\mf{n}^+$ is non-zero.
     To see this, note that given $X,Y \in \mf{p}$, we have that
     \[ \s(ad_X(Y)) = ad_{\s(X)}(\s(Y)) = ad_X(Y) \]
     Hence, $ad_X(Y) \in \mf{h}$.
     On the other hand, if $X\in \mf{a} \subseteq \mf{p}$ and $Y \in \mf{g}_\a$ for some $\a\neq 0\in \Sigma$,
     then $ad_X(Y) \in \mf{g}_\a$.
     In particular, this implies 
     \[ \mf{n}^+ \cap \mf{p} = \set{0} \]
     Thus, given any $X\neq 0\in \mf{n}^+$, the element $X + \s(X)\neq 0$ and is $\s$ invariant and hence belongs to $\mf{h}$.

     Next, note that since $\s(X) = -X$ for all $X\in \mf{a}$, we have $\s(\mf{n}^+) = \mf{n}^-$.
     Thus, in particular, for any $v\in \mf{n}^+$,
     \begin{align} \label{action on n+ and n-}
          \norm{Ad(g_n)(v)} \r \infty,   Ad(g_n)(\s(v)) \r 0
     \end{align} 

     Let $V = \set{v_\a \in \mf{g}_\a: \norm{v_\a} = 1, \a\in \Sigma^+}$ be fixed.
     For each $n$, let $v_{\a_n} \in V$ be such that 
     	\[\a_n(X_n) = \max\set{\a(X_n): \a\in \Sigma^+} \]
     where $X_n\in\mf{C}$ was such that $a_n = \exp(X_n)$.
     Now, for each $n$, let
     \begin{align} \label{definition of v_n}
     	v_n = \frac{v_{\a_n}+\s(v_{\a_n})}{\norm{Ad(g_n)}} 
     \end{align} 
     Then, for all $n$, $v_n\neq 0$ in $\mf{h}$ and satisfies
     \[ \norm{v_n} \ll \frac{1}{\norm{Ad(g_n)}}  \]
     Moreover, by the standard identity $Ad(\exp) = \exp(ad)$ and by~\ref{action on n+ and n-}, 
     \[ Ad(g_n)(v_n) =  \frac{e^{\a_n(X_n)}}{\norm{Ad(g_n)}}Ad(k_n)(v_{\a_n}) + o(1) \]
     
     By compactness of $K$, we have $ \norm{Ad(g_n)} \ll \norm{Ad(a_n)} $.
     But, by our choice of $\a_n$, $e^{\a_n(X_n)}$ is the largest eigenvalue of $Ad(a_n)$ and $Ad(a_n)$ is diagonalizable.
     Thus, $ e^{\a_n(X_n)}/\norm{Ad(a_n)} = O(1)  $ and so we get
     \[ \norm{v_{\a_n}} \ll \norm{Ad(g_n)(v_n)} \leq \norm{Ad(g_n)} \norm{v_n} \ll  \norm{v_{\a_n}} \]
      Hence, by passing to a subsequence, we get that
      \[ g_n\exp(v_n)g_n^{-1} = \exp(Ad(g_n)(v_n)) \r u \neq id \]

      Since $v_n \r 0$, we have that $\exp(v_n) \r id$. Hence, all the eigenvalues of $Ad(\exp(v_n))$ converge to $1$.
      Since conjugation doesn't change eigenvalues, we get that $u$ must be an $Ad$-unipotent element, which finishes the proof in the case $g_n \in KA$.

      For the general case, by $KAH$ decomposition, we write $g_n = k_na_nh_n$.
      Then, we can find $v_n\in \mf{h}$ as above such that $Ad(k_na_n)(\exp(v_n)) \r u \neq id$.
      Thus, the elements $w_n = h_n^{-1} v_n h_n$ will satisfy $Ad(g_n)(\exp(w_n)) = Ad(k_na_n)(\exp(v_n)) \r u$.
      By our assumption on the boundedness of $\norm{Ad(h_n^{-1})}$, we get
      \[ \norm{w_n} \leq \norm{Ad(h_n^{-1})} \norm{v_n} \ll \frac{1}{\norm{Ad(g_n)}} \]
      Hence, the sequence $w_n$ satisfies the conclusion of the Proposition.

\bibliography{bibliography}{}

\begin{thebibliography}{KSW}

\bibitem[BO]{BenoistOh}
Yves Benoist and Hee Oh.
\newblock {Effective equidistribution of S-integral points on symmetric
  varieties}.
\newblock {\em Annales de l'institut Fourier} {\bf 62}(2012), 1889--1942.

\bibitem[Bou]{BourgainArithmeticSets}
Jean Bourgain.
\newblock {Pointwise ergodic theorems for arithmetic sets}.
\newblock {\em Publications Mathématiques de l'IHÉS} {\bf 69}(1989), 5--41.

\bibitem[CE]{EskinChaika}
Jon Chaika and Alex Eskin.
\newblock {Every flat surface is birkhoff and oseledets generic in almost every
  direction}.
\newblock {\em Journal of Modern Dynamics} {\bf 9}(2015), 1--23.

\bibitem[EW]{EinsiedlerWard}
Manfred Einsiedler and Thomas Ward.
\newblock {\em Ergodic Theory}.
\newblock Springer Nature, 2011.

\bibitem[EM]{eskin1993}
Alex Eskin and Curt McMullen.
\newblock {Mixing, counting, and equidistribution in Lie groups}.
\newblock {\em Duke Math. J.} {\bf 71}(07 1993), 181--209.

\bibitem[EMS]{EskinMozesShah}
Alex Eskin, Shahar Mozes, and Nimish Shah.
\newblock {Unipotent Flows and Counting Lattice Points on Homogeneous
  Varieties}.
\newblock {\em Annals of Mathematics} {\bf 143}(1996), 253--299.

\bibitem[Gro]{GroemerConvexGeometry}
H.~Groemer.
\newblock {On the symmetric difference metric for convex bodies.}
\newblock {\em Beiträge zur Algebra und Geometrie} {\bf 41}(2000), 107--114.

\bibitem[KSW]{KleinbockShiWeiss}
Dmitry Kleinbock, Ronggang Shi, and Barak Weiss.
\newblock {Pointwise equidistribution with an error rate and with respect to
  unbounded functions}.
\newblock {\em Mathematische Annalen} {\bf 367}(2017), 857--879.

\bibitem[Moz]{Mozes1992}
Shahar Mozes.
\newblock {Mixing of all orders of Lie groups actions.}
\newblock {\em Inventiones mathematicae} {\bf 107}(1992), 235--242.

\bibitem[Rat]{RatnerMeasure}
Marina Ratner.
\newblock {On Raghunathan's Measure Conjecture}.
\newblock {\em Annals of Mathematics} {\bf 134}(1991), 545--607.

\bibitem[SU]{Sarnak2015}
Peter Sarnak and Adri{\'{a}}n Ubis.
\newblock {The horocycle flow at prime times}.
\newblock {\em Journal de Math{\'{e}}matiques Pures et Appliqu{\'{e}}es} {\bf
  103}(feb 2015), 575--618.

\bibitem[Sch]{Schlichtkrull}
Henrik Schlichtkrull.
\newblock {\em Hyperfunctions and Harmonic Analysis on Symmetric Spaces
  (Progress in Mathematics)}.
\newblock Birkhäuser, 1984.

\bibitem[Sha]{ShahPolynomial}
Nimish~A. Shah.
\newblock {Limit distributions of polynomial trajectories on homogeneous
  spaces}.
\newblock {\em Duke Math. J.} {\bf 75}(09 1994), 711--732.

\bibitem[Shi]{RShi-Pointwise}
Ronggang Shi.
\newblock {Pointwise equidistribution for one parameter diagonalizable group
  action on homogeneous space}, 2014.

\bibitem[Ven]{Venkatesh:Equi}
A.~Venkatesh.
\newblock {Sparse equidistribution problems, period bounds and subconvexity}.
\newblock {\em Annals of Math.} {\bf 172}(2010), 989--1094.

\end{thebibliography}
\bibliographystyle{math}

\end{document}